\documentclass[12pt]{amsart}
\usepackage{amsmath,amssymb,amsbsy,amsfonts,amsthm,latexsym,amsopn,amstext,
                            amsxtra,euscript,amscd}

\newfont{\teneufm}{eufm10}
\newfont{\seveneufm}{eufm7}
\newfont{\fiveeufm}{eufm5}
\newfam\eufmfam
                 \textfont\eufmfam=\teneufm \scriptfont\eufmfam=\seveneufm
                 \scriptscriptfont\eufmfam=\fiveeufm
\def\bbbc{{\mathchoice {\setbox0=\hbox{$\displaystyle\rm C$}\hbox{\hbox
to0pt{\kern0.4\wd0\vrule height0.9\ht0\hss}\box0}}
{\setbox0=\hbox{$\textstyle\rm C$}\hbox{\hbox
to0pt{\kern0.4\wd0\vrule height0.9\ht0\hss}\box0}}
{\setbox0=\hbox{$\scriptstyle\rm C$}\hbox{\hbox
to0pt{\kern0.4\wd0\vrule height0.9\ht0\hss}\box0}}
{\setbox0=\hbox{$\scriptscriptstyle\rm C$}\hbox{\hbox
to0pt{\kern0.4\wd0\vrule height0.9\ht0\hss}\box0}}}}
\def\bbbq{{\mathchoice {\setbox0=\hbox{$\displaystyle\rm
Q$}\hbox{\raise
0.15\ht0\hbox to0pt{\kern0.4\wd0\vrule height0.8\ht0\hss}\box0}}
{\setbox0=\hbox{$\textstyle\rm Q$}\hbox{\raise
0.15\ht0\hbox to0pt{\kern0.4\wd0\vrule height0.8\ht0\hss}\box0}}
{\setbox0=\hbox{$\scriptstyle\rm Q$}\hbox{\raise
0.15\ht0\hbox to0pt{\kern0.4\wd0\vrule height0.7\ht0\hss}\box0}}
{\setbox0=\hbox{$\scriptscriptstyle\rm Q$}\hbox{\raise
0.15\ht0\hbox to0pt{\kern0.4\wd0\vrule height0.7\ht0\hss}\box0}}}}
\def\bbbt{{\mathchoice {\setbox0=\hbox{$\displaystyle\rm
T$}\hbox{\hbox to0pt{\kern0.3\wd0\vrule height0.9\ht0\hss}\box0}}
{\setbox0=\hbox{$\textstyle\rm T$}\hbox{\hbox
to0pt{\kern0.3\wd0\vrule height0.9\ht0\hss}\box0}}
{\setbox0=\hbox{$\scriptstyle\rm T$}\hbox{\hbox
to0pt{\kern0.3\wd0\vrule height0.9\ht0\hss}\box0}}
{\setbox0=\hbox{$\scriptscriptstyle\rm T$}\hbox{\hbox
to0pt{\kern0.3\wd0\vrule height0.9\ht0\hss}\box0}}}}
\def\bbbs{{\mathchoice
{\setbox0=\hbox{$\displaystyle     \rm S$}\hbox{\raise0.5\ht0\hbox
to0pt{\kern0.35\wd0\vrule height0.45\ht0\hss}\hbox
to0pt{\kern0.55\wd0\vrule height0.5\ht0\hss}\box0}}
{\setbox0=\hbox{$\textstyle        \rm S$}\hbox{\raise0.5\ht0\hbox
to0pt{\kern0.35\wd0\vrule height0.45\ht0\hss}\hbox
to0pt{\kern0.55\wd0\vrule height0.5\ht0\hss}\box0}}
{\setbox0=\hbox{$\scriptstyle      \rm S$}\hbox{\raise0.5\ht0\hbox
to0pt{\kern0.35\wd0\vrule height0.45\ht0\hss}\raise0.05\ht0\hbox
to0pt{\kern0.5\wd0\vrule height0.45\ht0\hss}\box0}}
{\setbox0=\hbox{$\scriptscriptstyle\rm S$}\hbox{\raise0.5\ht0\hbox
to0pt{\kern0.4\wd0\vrule height0.45\ht0\hss}\raise0.05\ht0\hbox
to0pt{\kern0.55\wd0\vrule height0.45\ht0\hss}\box0}}}}
\def\bbbz{{\mathchoice {\hbox{$\sf\textstyle Z\kern-0.4em Z$}}
{\hbox{$\sf\textstyle Z\kern-0.4em Z$}}
{\hbox{$\sf\scriptstyle Z\kern-0.3em Z$}}
{\hbox{$\sf\scriptscriptstyle Z\kern-0.2em Z$}}}}

 \newtheorem{thm}{Theorem}%% [section]
 
 \newtheorem{lem}[thm]{Lemma}
 
 \theoremstyle{definition}
 
 \theoremstyle{remark}

\def\cA{{\mathcal A}}
\def\cB{{\mathcal B}}
\def\cC{{\mathcal C}}

\def\cG{{\mathcal G}}

\def\cI{{\mathcal I}}

\def\cQ{{\mathcal Q}}
\def\cR{{\mathcal R}}
\def\cS{{\mathcal S}}

\def\cX{{\mathcal X}}

\def\({\left(}
\def\){\right)}
\def\[{\left[}
\def\]{\right]}
\def\<{\langle}
\def\>{\rangle}

\def\fl#1{\left\lfloor#1\right\rfloor}
\def\rf#1{\left\lceil#1\right\rceil}

\def\F{\mathbb{F}}

\def\ep{{\mathbf{\,e}}_p}

\begin{document}

\title[Additive Decompositions of Subgroups of Finite Fields]{Additive Decompositions of Subgroups of Finite Fields}

\author{Igor E.~Shparlinski}
\address{Department of Computing, Macquarie University, Sydney, NSW 2109, Australia}

\email{igor.shparlinski@mq.edu.au}

\begin{abstract} 
We say that a set $\cS$ is additively decomposed into two sets $\cA$ and $\cB$,  
if $\cS = \{a+b~:~a\in \cA, \ b \in \cB\}$. 
Here we study additively decompositions of multiplicative subgroups of
finite fields. In particular, we give some improvements 
and generalisations  of results of C. Dartyge and A. S{\'a}rk{\"o}zy on 
additive decompositions of quadratic residues and primitive roots 
modulo $p$. We use some new tools such the Karatsuba bound of double 
character sums and some results from additive combinatorics. 
\end{abstract}



\maketitle

\section{Introduction}

Let $\F_q$ be the finite field of $q$ elements.
As usual, for two sets $\cA, \cB \subseteq \F_q$ we
define their sum as 
$$
\cA + \cB  = \{a+b~:~a\in \cA, \ b \in \cB\}.
$$
We say that a set  $\cS \subseteq \F_q$  is {\it additively decomposed\/}
into two sets if $\cS = \cA + \cB$. We say that an additively decomposition 
is nontrivial if 
$$
\min\{\#\cA, \# \cB\} \ge 2.
$$

S{\'a}rk{\"o}zy~\cite{Sark} has conjectured that the set $\cQ$ of quadratic 
residues modulo a prime $p$ does not have nontrivial decompositions and showed towards 
this conjecture that any nontrivial decomposition
$$
\cQ = \cA + \cB
$$
satisfies 
\begin{equation}
\label{eq:low Q}
\min\{\#\cA, \# \cB\} \ge \frac{p^{1/2}}{3\log p}
\end{equation}
and 
\begin{equation}
\label{eq:up Q}
\max\{\#\cA, \# \cB\} \le p^{1/2} \log p
\end{equation}

Furthermore, Dartyge and S{\'a}rk{\"o}zy~\cite{DaSa} have made a 
similar conjecture for the set $\cR$ of primitive roots modulo $p$ 
and given the following analogues of~\eqref{eq:low Q} and~\eqref{eq:up Q}:
\begin{equation}
\label{eq:low R}
\min\{\#\cA, \# \cB\} \ge \frac{\varphi(p-1)}{\tau(p-1)p^{1/2}\log p}
\end{equation}
and 
\begin{equation}
\label{eq:up R}
\max\{\#\cA, \# \cB\} \le \tau(p-1)p^{1/2}\log p
\end{equation}
where $\varphi(k)$ and $\tau(k)$ are the Euler function and the number 
of integer positive divisors of an integer $k \ge 1$. We also refer 
to~\cite{DaSa,Sark} for further references about set decompositions.

Here we consider  a more general question of additive decomposition
of arbitrary multiplicative subgroups $\cG \subseteq \F_q^*$.
In particular, our results for large subgroups leads to improvements
of the upper bounds~\eqref{eq:up Q}
and~\eqref{eq:up R}, which in turn imply an improvement of 
the lower bounds~\eqref{eq:low Q} and~\eqref{eq:low R}. 
These improvements are based on an application of 
a bound of  Karatsuba~\cite{Kar1} of double multiplicative character sums.
This technique work for subgroups of size or order at least $q^{1/2}$. 
For smaller subgroups in prime fields, that is, for prime $q=p$,  
we bring into this area yet another tool, 
coming from additive combinatorics. Namely, we use a result of
Garaev and Shen~\cite{GarShen}.
We also use a result of Shkredov and Vyugin~\cite{VyuShk} on the
size of the intetesection of shifts of a  multiplicative subgroup, to obtain 
an upper bound on the cardinalities $\#\cA$ and $\# \cB$ for additive decomposition
of arbitrary multiplicative subgroups $\cG \subseteq \F_p^*$.

Finally, we note that Shkredov~\cite{Shkr} has recently achieved remarkable
progress towards the conjecture of S{\'a}rk{\"o}zy~\cite{Sark}
and showed that the conjecture holds with $\cA=\cB$. That is,
$\cQ \ne \cA + \cA$ for any set $\cA \subseteq \F_p$. The method, however,
does not seem to extend to other subgroups. Shkredov~\cite{Shkr}
has also independently observed that $\log p$ 
can be removed from the bounds~\eqref{eq:low Q} and~\eqref{eq:up Q}
(which is a special case of Theorem~\ref{thm:Large Group} below). 

We recall that
the expressions $A \ll B$,  $B \gg A$ and $A=O(B)$ are each equivalent to the
statement that $|A|\le cB$ for some constant $c$. Throughout the paper, 
the implied constants in these symbols may depend on the real parameter $\varepsilon > 0$
and the integer parameter $\nu \ge 1$, and  are absolute otherwise.  

We also use the convention that for elements $\lambda, \mu \in \F_q$
and a set $\cS \subseteq \F_q$, 
$$
\lambda \cdot \cS + \mu = \{\lambda s + \mu~:~s \in \cS\},
$$
reserving, say, $2\cS$ for 
$$
2\cS = \cS+\cS.
$$

\section{Bounds of Multiplicative Character Sums}

We refer to~\cite{IwKow} for a background on multiplicative 
characters. 

First we recall the Weil bound of multiplicative character sums,
see~\cite[Theorem~11.23]{IwKow}.

\begin{lem}
\label{lem:Weil}
For any polynomial $F(X) \in \F_q[X]$ with  $D$ distinct zeros
in the algebraic closure of $\F_q$ and
which is not a perfect $d$th  and  any non-trivial
multiplicative character $\chi$ of $\F_q^*$ of order $d$, we
have
$$
\left| \sum_{x \in \F_q}\chi\(F(x)\)\right| \le (D-1) q^{1/2}.
$$
\end{lem}

We note that the following result is slightly more precise than 
a bound of  Karatsuba~\cite{Kar1} (see also~\cite[Chapter~VIII, Problem~9]{Kar2})
that applies to double character sums over arbitrary sets.

\begin{lem}
\label{lem:Kar} Let $\cA, \cB \subseteq \F_q$ be two arbitrary sets. 
For any non-trivial
multiplicative character $\chi$ of $\F_q^*$ and any
positive integer $\nu$, we have
$$
\sum_{a\in \cA} \sum_{b \in \cB} \chi(a+b) \ll 
\( \# \cA\)^{(2\nu-1)/2\nu}\( \(\# \cB\)^{1/2} q^{1/2\nu } + \# \cB q^{1/4\nu}\) . 
$$
\end{lem}

\section{A Bound on the Intersection of Shifted Subgroups}
\label{sec:Sols}

%Let $\cG\subseteq \F_q$  be the group of order $t = \# \cG$. 
%We denote by $N_t(b)$ the number of solution to the equation 
%$$u - v = b, \qquad u,v\in \cG.
%$$ 

Let us consider a multiplicative subgroup $\cG\subseteq \F_q^*$.
The following bound  for $m=1$  is due to Garcia and Voloch~\cite{GaVo}, 
see also~\cite{HBK,Kon}. 
 For a fixed $m \ge  1$ it follows instantly
from a result of Shkredov and Vyugin~\cite[Corollary~1.2]{VyuShk} by taking  $s=1$, $t= \# \cG$, $k = m-1$
and $B = \fl{(\# \cG)^{1/(2k+1)}}+1$.

\begin{lem}
\label{lem:ShkVyu} Assume that for a fixed integer $m\ge 2$ we have
a prime $p$ satisfies:
$$
p \ge 4(m-1) \# \cG\((\# \cG)^{1/(2m-1)}+ 1\) .
$$
Then for  pairwise distinct  $b_1,\ldots, b_m \in \F_p^*$ the bound
$$
\# \( \bigcap_{i=1}^m  \(\cG + b_i\) \)
\le 4m \((\# \cG)^{1/(2m-1)}+1\)^m
$$
holds. 
\end{lem}

%and by $T_d$ the number of 
%solutions to the equation 
%$$
%u_1+u_2 = u_3+u_4, \qquad u_1, u_2, u_3, u_4 \in \cG_d.
%$$
%
%We recall the bound of Garcia and Voloch~\cite{GaVo},
%see also~\cite{HBK,Kon,VyuShk}.

%

%\begin{lem}
%\label{lem:GaVo} For a prime $q = p$ and  $d \mid p-1$
%with $t \le p^{3/4}$, for any $b\in \F_p^*$, we have 
%$$
%N_t(b) \ll \(\#\cG\)^{2/3}.
%$$
%\end{lem}

%Heath-Brown \&  Konyagin~\cite{HBK} have proved that for a prime $q=p$
%and $d \mid p-1$ with $d \gg p^{1/3}$
%we have 
%$$
%T_d \ll \(\#\cG_d\)^{5/2}, 
%$$
% see also~\cite{Kon} for a generalisation. 
%For small subgroups 
%Shkredov~\cite{Shkr2} gives a stronger estimate

%
%\begin{lem}
%\label{lem:Shkr} For a prime $q = p$ and  $d \mid p-1$
%with $d \gg p^{2/5} \(\log p\)^{6/5}$ we have 
%$$
%T_d \ll \(\#\cG_d\)^{22/9}  \(\log \#\cG_d\)^{2/3}.
%$$
%\end{lem}

\section{A Result from Additive Combinatorics}

We extend in a natural way our definition of the sums set $\cA+\cB$ to
other operations on sets. 
For example, for $\cA\subseteq \F_q$ we have
$$
\cA^2 =  \{a_1a_2~:~a_1,a_2\in \cA\}
$$
and
$$
\cA(\cA + 1) = \{a_1(a_2+1)~:~a_1,a_2\in \cA\}.
$$

%We recall that  Rudnev~\cite[Theorem~1]{Rud}, improving several 
%previous results, more remarkably that of  Bourgain and Garaev~\cite{BouGar}, 
%Garaev~\cite{Gar1} and Katz and Shen~\cite{KatzShen},  has established the inequality

%\begin{lem}
%\label{lem:AA A+A}
%For any prime $p$ and set $\cA\subseteq \F_p$ of size 
%$ \# \cA\le \sqrt{p}$ 
%we have
%$$
%\max\{\#\(\cA^2\), \#\(2\cA\)\} \ge (\# \cA)^{12/11 + o(1)}.
%$$
%\end{lem}

%We also recall the following result, which is a 
%variant  the {\it Pl{\"u}nnecke inequality\/}, see~\cite[Lemma~2.7]{BouGar}, 
%or~\cite[Lemma~2]{Gar1}, or~\cite[Corollary~2.2]{Gar2},  or~\cite[Corollary~1.4]{KatzShen}.

%\begin{lem}
%\label{lem:Plun}
%For any prime $p$ and $k \ge 2$ sets $\cA_1, \ldots \cA_k, \cB \subseteq \F_p$ we have
%$$
% \#\(\cA_1 + \ldots + \cA_k\) \le \frac{1}{(\# \cB)^{k-1}} \prod_{i=1}^k \#(\cA_i + \cB).
%$$
%\end{lem}

%In particular, for $k=2$ and $\cA_1=\cA_2 = \cA$ we obtain:

%\begin{cor}
%\label{cor:Triangle}
%For any prime $p$ and $k \ge 2$ sets $\cA_1, \ldots \cA_k, \cB \subseteq \F_p$ we have
%$$
% \#\(2\cA\) \#\cB \le \(\#(\cA + \cB)\)^2.
%$$
%\end{cor}

We also need the following combination of two results of 
 Garaev and Shen~\cite[Theorems~1 and~2]{GarShen}.

\begin{lem}
\label{lem:A(A+1)}
For any $\varepsilon>0$ there exists some $\delta>0$ 
such that for any prime $p$ and  set $\cA\subseteq \F_p$ of size 
$ \# \cA\le p^{1-\varepsilon}$ 
we have
$$
 \#\(\cA(\cA + 1)\) \gg (\# \cA)^{57/56 + o(1)}.
$$
\end{lem}

We remark that for sets f size 
$ \# \cA\le p^{1/2}$ 
of Jones  and Roche-Newton~\cite{JonRoch}, 
have improved~\cite[Theorem~1]{GarShen}, however this 
is not essential for our final estimate. 
Furthermore, for sets with $ p^{\varepsilon} \# \cA\le p^{1-\varepsilon}$ 
the estimate of Lemma~\ref{lem:A(A+1)} is also given by Bourgain~\cite{Bour1}.

\section{Preliminary Estimates}
\label{sec:Prelim}

Let $\cG_d\subseteq \F_q$  be the group of $d$th powers. 
We start with the following generalisation of the bound~\eqref{eq:low Q},
which closely follows the arguments of~\cite{DaSa,Sark}. 

\begin{lem}
\label{lem:Gd low} 
Let $d\mid q-1$. Then  for
any nontrivial decomposition of $\cG_d$ into some sets $\cA$ and $\cB$, we have
$$
\min\{\#\cA, \# \cB\} \ge  \(2  + o(1)\) \frac{q^{1/2} \log d}{d^2\log q} 
$$
as $q\to \infty$. 
\end{lem}  

\begin{proof}  Let $K =\#\cA$ and $L = \# \cB$. 
Assume that $K\ge L$. Maybe an additive shift 
by an element of $b\in \cB$, that is considering $\cA + a$ and $\cB-b$ 
we can assume that $0 \in \cB$ and thus $\cA \subseteq \cG_d$. 

First we show that 
\begin{equation}
\label{eq:small B}
L \ge  \(\frac{1}{\log 2} + o(1)\) \log (q^{1/2}/d). 
\end{equation}
Indeed, for any $u \in \cG_d$  we have 
\begin{equation}
\label{eq:Gd-u Gd}
u - b \in \cA \subseteq \cG_d
\end{equation}
for at least one $b \in \cB$. We now show that if~\eqref{eq:small B}
does not hold then this is impossible. 

Let $\cX_d$ be the set of all multiplicative characters of $\F_q^*$
with $\chi^d = \chi_0$ where $\chi_0$ is the principal character.
We also define $\cX_d^* = \cX_d \setminus \{\chi_0\}$.
We note that $v \in \cG_d$ if and only if
$$
\frac{1}{d}\sum_{\chi \in \cX_d} \chi(v) = \left\{\begin{array}{ll}
1,& \text{if $v \in \cG_d$,}\\
0,&\text{otherwise}.
\end{array}
\right.$$
%we refer to~\cite{IwKow} for a background on multiplicative 
%characters. 
Therefore, the condition~\eqref{eq:Gd-u Gd} implies
$$
\prod_{b\in \cB} \(1 - \frac{1}{d}\sum_{\chi \in \cX_d} \chi(u - b)\)=0.
$$
Since elements of $\cG_d$ are all of the form $x^d$, $x \in \F_q^*$, 
we see that the sum 
$$
W =\sum_{x \in \F_q^*} \prod_{b\in \cB} \(1 - \frac{1}{d}\sum_{\chi \in \cX_d} \chi(x^d - b)\)
$$
vanishes.

On the other hand, 
separating the contribution of the principal characters, we write
\begin{equation}
\begin{split}
\label{eq:W exp}
W & =\sum_{x \in \F_q^*} \prod_{b\in \cB} \(\frac{d-1}{d} - \frac{1}{d}\sum_{\chi \in \cX_d^*} \chi(x^d - b)\)\\
& = (q-1)\(\frac{d-1}{d}\)^{L} + R,
 \end{split}
\end{equation}
where, after the change of order of summation we obtain
\begin{equation}
\begin{split}
\label{eq:R exp}
R & =  \frac{1}{d^L} \sum_{\substack{\cC \subseteq \cB\\\cC \ne \emptyset}}
 \(-1 \)^{\# \cC} \(d -1 \)^{L-\# \cC}  \sum_{x \in \F_q^*} \prod_{c \in \cC} \sum_{\chi \in \cX_d^*} \chi(x^d - c)\\
& =  \frac{1}{d^L}
 \sum_{\ell=1}^{L} \(-1 \)^{\ell} \(d -1 \)^{L-\ell}  \sum_{\substack{\cC \subseteq \cB\\ \# \cC =\ell}}
  \sum_{x \in \F_q^*} \prod_{c \in \cC} \sum_{\chi \in \cX_d^*} \chi(x^d - c) .
 \end{split}
\end{equation}
For every set $\cC= \{c_1, \ldots, c_\ell\}$ of cardinality $\ell$, we have 
\begin{equation}
\label{eq:C exp}
\sum_{x \in \F_q^*} \prod_{c \in \cC} \sum_{\chi \in \cX_d^*} \chi(x^d - c)
 =  \sum_{\chi_1, \ldots, \chi_\ell \in \cX_d^*}\sum_{x \in \F_q^*} \prod_{i=1}^\ell  \chi_i(x^d - c_i).
\end{equation}
Clearly Lemma~\ref{lem:Weil}  applies to the inner sum and yields
$$
\left| \sum_{x \in \F_q^*} \prod_{i=1}^\ell  \chi_i(x^d - c_i)\right| <  d \ell q^{1/2}.
$$
Hence
$$
\left|\sum_{\chi_1, \ldots, \chi_\ell \in \cX_d^*}\sum_{x \in \F_q^*} \prod_{i=1}^\ell  \chi_i(x^d - c_i)\right| < d (d-1)^\ell \ell q^{1/2}.
$$
Thus, substituting this bound in~\eqref{eq:C exp}, and then the resulting estimate 
in~\eqref{eq:R exp}, we obtain
$$
R   < d q^{1/2}\(\frac{d-1}{d}\)^{L} \sum_{\ell=1}^{L} \binom{L}{\ell}  \ell=  
dL  2^{L-1} \(\frac{d-1}{d}\)^{L} q^{1/2}.
$$
Recalling~\eqref{eq:W exp}, we derive
$$
\left|W -(q-1)\(\frac{d-1}{d}\)^{L} \right| < dL  2^{L-1} \(\frac{d-1}{d}\)^{L} q^{1/2}.
$$
Therefore, if $W = 0$ then 
$q-1 <  dL  2^{L-1} q^{1/2}$
and~\eqref{eq:small B} follows. 

For $d=2$ the result follows from a direct generalisation of~\eqref{eq:low Q}
to arbitrary finite fields.

We now assume that $d \ge 3$, as otherwise there is nothing to prove and set  
$$
L^* = \rf{ \frac{\log (q^{1/2}/d)}{\log d} }.
$$ 
Next, we 
choose an arbitrary subset $\cB^* \subseteq \cB$ of cardinality $L^*$ (which by~\eqref{eq:small B} 
is possible for a sufficiently large $p$).   

We denote by  $N$ is the number of $u \in \cG_d$ satisfying~\eqref{eq:Gd-u Gd} for every $b \in \cB^*$.
Clearly $N \ge K$. On the other hand, as before, we write
\begin{equation*}
\begin{split}
N & = \sum_{u \in \cG_d} \prod_{b\in \cB^*}\(\frac{1}{d}\sum_{\chi \in \cX_d} \chi(u - b)\)
= \frac{1}{d} \sum_{x \in \F_q^*} \prod_{b\in \cB^*}\(\frac{1}{d}\sum_{\chi \in \cX_d} \chi(x^d - b)\)\\
& =  \frac{1}{d^{L^*+1}} \sum_{x \in \F_q^*} \prod_{b\in \cB^*} \sum_{\chi \in \cX_d} \chi(x^d - b).
\end{split}
\end{equation*}
Exactly the same argument  as in the previous estimation of $W$ implies
\begin{equation*}
\begin{split}
\left|N -  \frac{(q-1)}{d^{L^*+1}}\right|
& \le   \frac{1}{d^{L+1}} \left|\sum_{x \in \F_q^*} \prod_{b\in \cB^*}  
\sum_{\chi \in \cX_d^*} \chi(x^d - b)\right|\\
& = \frac{dq^{1/2}}{d^{L^*+1}}\sum_{\substack{\cC \subseteq \cB^*\\\cC \ne \emptyset}} 
(d-1)^{\#\cC} \#\cC
= \frac{q^{1/2}}{d^{L^*}} \sum_{\ell=1}^{L^*} \binom{L^*}{\ell} (d-1)^{\ell} \ell   \\ 
& = \frac{q^{1/2}}{d^{L^*}} L^*d^{L^*-1}(d-1) <  L^* q^{1/2} .
\end{split}
\end{equation*}
Recalling the choice of $L^*$ we immediately derive
\begin{equation}
\label{eq:up G prelim}
K \le N \le \( \frac{1}{2} + o(1)\)  \frac{  d q^{1/2} \log q}{\log d}.
\end{equation}
Since we obviously have $KL \ge \#\cG_e = (q-1)/d$ the result follows. 
\end{proof}

We now use Lemma~\ref{lem:A(A+1)} to study nontrivial 
additive decompositions of small subgroups.

\begin{lem}
\label{lem:Small Group} For any $\varepsilon>0$ there exists some $\kappa>0$ 
such that if for a prime $q=p$ and
a   subgroup   $\cG \subseteq\F_p^*$ of order  $\# \cG < p^{1-\varepsilon}$
there is a  nontrivial decomposition into some sets $\cA$ and $\cB$, then 
$$
\(\#\cG\)^{\kappa}\le  \min\{\#\cA, \# \cB\} \le  \max\{\#\cA, \# \cB\} 
 \ll \(\#\cG\)^{1-\kappa}.
$$
\end{lem}

\begin{proof}    Assume that $\# \cA \ge \# \cB$. Also, 
as in the proof of  Lemma~\ref{lem:Gd low}, 
we see that we can assume that $\#\cA \subseteq \cG$. 

Since $\#\cB \ge 2$, there is $b \in \cB$ with $b \ne 0$.
Then, from the identity
$$
\cA(\cA+b) = b^2 (b^{-1}\cA)(b^{-1}\cA+1)
$$
and Lemma~\ref{lem:A(A+1)} we see that $\#\(\cA(\cA+b)\) \ge (\# \cA)^{1+\delta}$
for some $\delta > 0$ that depends only on $\varepsilon$. 
On the other hand, we obviously have $\cA(\cA+b) \subseteq \cG$, 
thus $\# \cA \le (\# \cG)^{1/(1 +\delta)}$. Since $\#\cA \# \cB\ge \# \cG$,
the result now follows. 
\end{proof}

\section{Decompositions of Large Multiplicative Subgroups and the Set  of Primitive Roots}

Clearly the bound~\eqref{eq:up G prelim} is of the same order of magnitude 
as~\eqref{eq:up Q}. Here we use Lemma~\ref{lem:Kar} to generalise and 
improve~\eqref{eq:up Q} and~\eqref{eq:up R}.

\begin{thm}
\label{thm:Large Group} For any $\varepsilon > 0$, 
if for a   subgroups   $\cG \subseteq\F_q^*$ of order  
$\# \cG \ge q^{3/4 +\varepsilon}$
or the set of primitive roots $\cR\subseteq\F_q^*$
there is a  nontrivial decomposition into some sets $\cA$ and $\cB$, then 
$$
 \max\{\#\cA, \# \cB\} \ll q^{1/2}.
$$
\end{thm} 

\begin{proof}  Clearly, if $\cR = \cA+\cB$ is a decomposition of the 
set of primitive roots, then multiplying each elements of $\cA$ and $\cB$ by 
a fixed quadratic non-residue $\xi$ we obtain 
$$
\xi\cdot \cA + \xi \cdot \cB \subseteq \cQ.
$$

Let $d = (q-1)/\# \cG$, thus $\cG = \cG_d$ in the notation of 
Section~\ref{sec:Prelim}.

We also remark that  $d \ll q^{1/4 - \varepsilon}$, so by Lemma~\ref{lem:Gd low} 
we have 
\begin{equation}
\label{eq:low G}
\min\{\#\cA, \# \cB\} \ge q^{\varepsilon}
\end{equation}
for any nontrivial decomposition of $\cG_d$. 
Furthermore, one can check that the bound~\eqref{eq:low R} can be extended to arbitrary 
finite fields, so~\eqref{eq:low G} also holds 
for the sets in any nontrivial decomposition of $\cR$.

Thus, it is enough to show that  
any  sets $\cA$ and $\cB$ with~\eqref{eq:low G}
such that $\cA + \cB \subseteq \cG_d$ we have
\begin{equation}
\label{eq:up G}
\max\{\#\cA, \# \cB\} \ll q^{1/2}.
\end{equation}

Assume that $\#\cA \ge \# \cB$. 

Now, let $\chi\in \cX_d^*$ be any non-principal   character of $\F_q^*$.
If $\cA + \cB \subseteq \cG_d$ then
we have 
$$
\sum_{a\in \cA} \sum_{b \in \cB} \chi(a+b) =   \#\cA \# \cB.
$$
Comparing this with the bound of Lemma~\ref{lem:Kar},
we derive
\begin{equation}
\label{eq:Sum Low}
 \#\cA \# \cB \ll 
\( \# \cA\)^{(2\nu-1)/2\nu}\( \(\# \cB\)^{1/2} q^{1/2\nu } + \# \cB q^{1/4\nu}\) .
\end{equation}
Taking $\nu = \rf{\varepsilon^{-1}}$, we see that the condition~\eqref{eq:low G}
implies 
$$
\(\# \cB\)^{1/2} q^{1/2\nu } \le \# \cB q^{1/4\nu}.
$$
Hence~\eqref{eq:Sum Low} can now be re-written as
$$
 \#\cA \# \cB \ll \( \# \cA\)^{(2\nu-1)/2\nu}  \# \cB q^{1/4\nu},
$$
which implies~\eqref{eq:up G}, and  concludes the proof.
\end{proof}

Obviously, if $\cG = \cA+\cB$ then $\#\cG  \le   \#\cA \# \cB$.
Hence Theorem~\ref{thm:Large Group} implies that any  
nontrivial decomposition of a subgroup $\cG \subseteq\F_p^*$ 
of order  $\# \cG \ge q^{3/4 +\varepsilon}$
 into some sets $\cA$ and $\cB$, we have
$$
\min\{\#\cA, \# \cB\} \gg \#\cG q^{-1/2} 
$$
that is stronger than Lemma~\ref{lem:Gd low} 
and for $\cG = \cQ$ improves the bound~\eqref{eq:up Q} of 
S{\'a}rk{\"o}zy~\cite{Sark}.

Similarly,  any  
nontrivial decomposition of $\cR$ into some sets $\cA$ and $\cB$, we have
$$
\min\{\#\cA, \# \cB\} \gg \# \cR q^{-1/2} = \varphi(q-1) q^{-1/2},
$$
that improves the bound~\eqref{eq:up R} of  
Dartyge and S{\'a}rk{\"o}zy~\cite{DaSa}.

\section{Decompositions of Small Multiplicative Subgroups}

We now use Lemma~\ref{lem:A(A+1)} to study nontrivial 
additive decompositions of small subgroups of prime fields
that is, subgroups $\cG \subseteq\F_p^*$  with a prime $p$ 
of cardinality  $\#\cG$ to which the bound 
of Theorem~\ref{thm:Large Group} is either weak or does not 
apply (for example for subgroups of order  $\# \cG < p^{3/4}$).

\begin{thm}
\label{thm:Small Group} Let $q=p$ be prime.  
If for a   subgroup  $\cG \subseteq\F_p^*$ 
there is a  nontrivial decomposition into some sets $\cA$ and $\cB$, then 
$$
\(\#\cG\)^{1/2 + o(1)}= \min\{\#\cA, \# \cB\} \le  \max\{\#\cA, \# \cB\} 
= \(\#\cG\)^{1/2 + o(1)}, 
$$
as $\#\cG \to \infty$.
\end{thm}

\begin{proof}    Assume that $\# \cA \ge \# \cB$. 

Since $\#\cA \# \cB\ge \# \cG$ it suffices to only  establish 
the upper bound. In particular, it is enough to show that for an 
arbitrary  $\eta >0$ we
have 
\begin{equation}
\label{eq:A eps}
 \#\cA  \le \(\#\cG\)^{1/2 +\eta} .
\end{equation}

Let us fix some sufficiently small $\eta >0$. In particular, we assume 
that  that $\eta < 1/6$, thus $1/(1+2 \eta) > 3/4$. 
Then for $\#\cG \ge p^{1/(1+2 \eta)}$
the bound~\eqref{eq:A eps} follows from Theorem~\ref{thm:Large Group}.

So we now assume that $\#\cG < p^{1/(1+2 \eta)}$.
Then clearly for any $m\ge 2$ and a sufficiently large $p$  
the condition of Lemma~\ref{lem:ShkVyu} is satisfied. 

By Lemma~\ref{lem:Small Group}, the set $\# \cB$ is large enough, so that it 
has $m$ distinct elements $b_1, \ldots, b_m$.
We now observe that for every $i=1, \ldots, m$ we have 
$\cA \subseteq \cG - b_i$. Thus taking $m$ sufficiently large
we see that  Lemma~\ref{lem:ShkVyu}  implies~\eqref{eq:A eps}. 
Since $\#\cA \# \cB\ge \# \cG$,
the result now follows. 
\end{proof}

Clearly one can choose $m$ as a growing function of $\# \cG$ and
get more explicit bounds in Theorem~\ref{thm:Small Group}.

\section{Comments}

We remarks that it is natural to try to obtain analogues of our results for 
the dual problem of nontrivial multiplicative decompositions of intervals 
in $\F_p$.  That is, one can consider 
representations  $\cI = \cA\cB$ of sets $\cI= \{m+1, \ldots, m+n\} \subseteq \F_p$
of $n$ consecutive residues modulo $p$  with two arbitrary sets $\cA, \cB \subseteq \F_p^*$
such that  $\# \cA \ge \# \cB \ge 2$. Certainly various conjectures, similar to 
those of~\cite{DaSa,Sark}, can be about the non-existence
of such decompositions.  Here we merely show that
some of the above methods apply to multiplicative decompositions too.
For example, by a result of Bourgain~\cite{Bour2}, 
for any two sets $\cA, \cB \subseteq \F_p$ with $\cA, \cB \ne \{0\}$, we have
$$
\#\(8 \cA\cB - 8\cA\cB\) >\frac{1}{2}  
\min\left\{\#\cA \# \cB, p - 1\right\}.
$$
On the other hand, if $   \cA \cB =\cI$ then 
$$\#\(8 \cA\cB - 8\cA\cB\) \le 16 \#\cI.
$$
Thus for $\#\cI < (p-1)/32$  we derive very tight bounds:
$$
\#\cI \le \#\cA \# \cB  \le  32 \#\cI.
$$

Now we note, as in the above we can assume that 
$1 \in \cB$, so $\cA \subseteq \cI$.
Using the orthogonality of the exponential function
$$
\ep(z) = \exp(2 \pi i z/p)
$$
we can write the number of solutions $J$ to the congruence 
$$
u \equiv ab \pmod p, \qquad u \in \cI, \ a\in \cA, \ b \in \cB,
$$
as 
$$
J = \sum_{u \in \cI} \sum_{a\in \cA} \sum_{b \in \cB}
\frac{1}{p} \sum_{-(p-1)/2 \le \lambda \le (p-1)/2}
\ep(\lambda(u - ab)).
$$
Changing the order of summation and separating the contribution 
$ \# \cA \#\cB \# \cI/p$ of terms corresponding to $\lambda=0$,
we obtain
$$
\left| J - \frac{ \# \cA \#\cB \# \cI}{p}\right|\le
\frac{1}{p} \sum_{1 \le |\lambda| \le (p-1)/2}
\left|\sum_{u \in \cI}\ep(\lambda u)\right|
\left| \sum_{a\in \cA} \sum_{b \in \cB}
\ep(\lambda ab)\right|.
$$
Using the classical estimate
$$
\left| \sum_{a\in \cA} \sum_{b \in \cB}
\ep(\lambda ab)\right| \le (p \# \cA \#\cB)^{1/2},
$$
see, for example, see~\cite[Equation~(1.4)]{BouGar}
or~\cite[Lemma~4.1]{Gar2}, together with the bound
$$
\left|\sum_{u \in \cI}\ep(\lambda u)\right| \le \frac{p}{|\lambda|}, 
$$
that holds for $1 \le |\lambda| \le (p-1)/2$
see~\cite[Bound~(8.6)]{IwKow}. We derive
$$
J - \frac{ \# \cA \#\cB \# \cI}{p} \ll 
(p \# \cA \#\cB)^{1/2} \log p.
$$
On the other hand, since $\cI = \cA \cB$ is multiplicative decomposition of $\cI$,
we have $J = \# \cA \#\cB$. Thus for any fixed $\varepsilon > 0$,  
if  $\# \cI < (1-\varepsilon) p$ then 
%\begin{equation}
%\label{eq:Prod AB p}
$$
 \# \cA \#\cB \ll p(\log p)^2.
$$

One can also consider more general questions about sets of the 
form
$$
F(\cA, \cB) = \{F(a,b)~:~a\in \cA, \ b \in \cB\}
$$
with $F(X,Y) = \F_q[X,Y]$, representing subgroups and 
intervals.

Finally, we note that  Gyarmati, Konyagin and S{\'a}rk{\"o}zy~\cite{GyKoSa}
have studied additive decompositions of large subsets (of size 
close to $p$) of prime fields $\F_p$. The Weil bound of 
multiplicative character sums also plays a prominent role in
the arguments of~\cite{GyKoSa}. 
Several more results about decompositions of arbitrary sets can be 
found in~\cite{Alon,AGU}.

\section{Acknowledgment}

The author is grateful to  Cecile Dartyge for a careful reading of 
the manuscript and many valuable comments, 
 to Sergei Konyagin for his suggestion to use
the result of Shkredov and Vyugin~\cite{VyuShk} in the proof of
Theorem~\ref{thm:Small Group} and to Ilya Shkredov for useful
discussions and providing a preliminary version of~\cite{Shkr}. 

This work was supported in part by the  ARC Grant DP130100237.

\end{document}